\documentclass{amsart}

\usepackage{amsmath,amssymb,amsthm, mathrsfs, mathtools}
\usepackage{amscd,mathtools}
\usepackage[utf8]{inputenc}
\usepackage[english]{babel}
\usepackage{cite}
\usepackage{color}
\usepackage{hyperref}
\usepackage{float}
\usepackage{tikz}
\usetikzlibrary{automata}
\usepackage{comment}
\usepackage{mathtools}
\usepackage{csquotes}
\usepackage[margin=3cm]{geometry}
\newcommand*{\rom}[1]{\expandafter\@slowromancap\romannumeral #1@}

\usepackage{graphicx}
\usepackage{caption}
\usepackage{subcaption}
\usepackage[T1]{fontenc}
\usepackage[utf8]{inputenc}
\usepackage{mathtools}

\usepackage{amsthm}

\usepackage{tikz}
\usetikzlibrary{calc}
\usetikzlibrary{arrows}
\usetikzlibrary{decorations.pathreplacing}
\usetikzlibrary{intersections}

 \newtheorem{theorem}{Theorem}[section]
  
  \newtheorem{corollary}[theorem]{Corollary}
  \newtheorem{lemma}[theorem]{Lemma}

  \newtheorem{introthm}{Theorem}
  \newtheorem{introcor}[introthm]{Corollary}

  \theoremstyle{definition}
  \newtheorem{definition}[theorem]{Definition}

  \newtheorem*{claim*}{Claim}

    \newtheorem{notation}[theorem]{Notation}

  \newtheorem*{question*}{Question}
  \newtheorem*{answer*}{Answer}
  \newtheorem*{application*}{Application}

  \theoremstyle{remark}
  \newtheorem{remark}[theorem]{Remark}
  \newtheorem*{remark*}{Remark}
\makeatletter
\newcommand{\RNum}[1]{\uppercase\expandafter{\romannumeral #1\relax}}
\usepackage{graphicx}
\usepackage{caption}
\usepackage{subcaption}

\usepackage{amssymb}

\DeclarePairedDelimiterX{\Norm}[1]{\lVert}{\rVert}{#1}
 \newcommand{\from}{\colon\thinspace} 
\theoremstyle{definition}

  \newcommand{\sC}{{\sf C}}

    \newcommand{\E}{{\mathbb{e}}}

  \newcommand{\mm}{{\sf m}}   
  \newcommand{\nn}{{\sf n}}

  \newcommand{\gothic}{\mathfrak}
  \newcommand{\go}{{\gothic o}}

  \newcommand{\calN}{\mathcal{N}}

  \newcommand{\ST}{\mathbin{\Big|}} 
 
\begin{document}

\title{Convergence of sublinearly contracting horospheres}

  \author   {Abdul Zalloum}
 \address{Department of Mathematics, Queen's University, Kingston, ON }
 \email{az32@queensu.ca}
\begin{abstract}
   In \cite{QR19}, Qing, Rafi and Tiozzo introduced the sublinearly contracting boundary for CAT(0) spaces. Every point of this boundary is uniquely represented by a sublinearly contracting geodesic ray: a geodesic ray $b$ where every disjoint ball projects to a subset whose diameter is bounded by a sublinear function in terms of the ball's distance to the origin. This paper analyzes the bahaviour of horofunctions associated to such geodesic rays, for example, we show that horospheres associated to such horofunctions are convergent. As a consequence of this analysis, we show that for any proper CAT(0) space $X$, every point of the visual boundary $\partial X$ that is defined by a sublinearly contracting geodesic ray is a visibility point.

\end{abstract}

\maketitle

\section{Introduction}

Recently, Qing, Rafi and Tiozzo \cite{QR19} introduced the notion of sublinearly contracting boundaries for CAT(0) spaces and were able to show that for a proper CAT(0) space, such a boundary is a metrizable topological space which is invariant under quasi-isometries. For a sublinear function $\kappa$, a geodesic ray $c$ is said to be \emph{$\kappa$-contracting} if there exists a constant $\nn \geq0$ such that for any ball $B$ centered at $x$ and disjoint from $b$, we have $diam(\pi_b(B)) \leq \nn \kappa(||x||),$ where $||x||=d(x, b(0))$. A geodesic ray is \emph{sublinearly contracting} if it's $\kappa$-contracting for some sublinear function $\kappa$. Given a geodesic ray $c$, a \emph{horofunction} associated to such a geodesic ray, denoted by $h,$ is given by $h(x)=\underset{t \rightarrow \infty}{\lim}[d(x,c(t))-t] +\nn$ for some constant $\nn \geq 0$. An \emph{$h$-horoball} is defined to be the set $h^{-1}((-\infty,-k])$ for some $k \in \mathbb{R}^{+}.$ Similarly, an \emph{$h$-horosphere}, denoted by $H_k,$ is the set $h^{-1}(-k)$ for some $k \in \mathbb{R}^{+}.$ Let $h$ be a horofunction associated to the geodesic ray $c$ and let $\zeta=c(\infty) \in \partial X;$ the visual boundary of $X.$ A sequence of $h$-horospheres $\{H_k\}_{k \in \mathbb{N}}$ is said to be \emph{convergent} if whenever $x_k \in H_k,$ we have $x_k \rightarrow \zeta.$ See Figure \ref{fig: convergence of horospheres}.

\begin{figure}\label{fig: convergence of horospheres}
    \centering

\begin{tikzpicture}[scale=.3]

\draw[very thick,black] (0,0) circle (3cm);

\draw[thick, red, ->] (0.05,0) -- (0.05,3);

\draw[very thick,blue] (0,2.6) circle (.4cm);
\draw[very thick,blue] (0,2.3) circle (.7cm);
\draw[very thick,blue] (0,2) circle (1cm);
\draw[very thick,blue] (0,1.7) circle (1.3cm);
\draw[very thick,blue] (0,1.4) circle (1.6cm);
\draw[very thick,blue] (0,1.1) circle (1.9cm);

\draw[thick,fill=black] (0,3) circle (.1cm);

\node[above] at (.5,3) {$\zeta$};
\node[below] at (.5,-3) {$\mathbb{H}^2$};

\end{tikzpicture}
\begin{tikzpicture}[scale=.2]
\draw[thin,<->] (-5,0) -- (9,0);
\draw[thin,->] (0,0) -- (0,9);

\draw[thick, red, ->] (0,0) -- (0,9);
\node[above] at (0,9) {$\eta$};

\draw[thin, blue, <->] (-5,1) -- (9,1);
\draw[thin, blue, <->] (-5,2) -- (9,2);
\draw[thin, blue, <->] (-5,3) -- (9,3);
\draw[thin, blue, <->] (-5,4) -- (9,4);
\draw[thin, blue, <->] (-5,5) -- (9,5);
\draw[thin, blue, <->] (-5,6) -- (9,6);
\draw[thin, blue, <->] (-5,7) -- (9,7);
\node[below] at (0,0) {$\mathbb{R}^2$};

\draw [fill=green] (0,0) circle (.2cm);
\draw [fill=green] (1,1) circle  (.2cm);
\draw [fill=green] (2,2) circle  (.2cm);

\draw [fill=green] (3,3) circle  (.2cm);
\draw [fill=green] (4,4) circle  (.2cm);
\draw [fill=green] (5,5) circle  (.2cm);

\draw [fill=black] (-2,0) circle  (.2cm);
\draw [fill=black] (-2,1) circle  (.2cm);
\draw [fill=black] (-2,2) circle (.2cm);
\draw [fill=black] (-2,3) circle (.2cm);
\draw [fill=black] (-2,4) circle (.2cm);
\draw [fill=black] (-2,5) circle (.2cm);

\draw[thin, white, ->] (-10,5) -- (-15,4);

\end{tikzpicture}

\caption{The blue horospheres associated to the point $\zeta$ in the boundary of $\mathbb{H}^2$ are convergent . On the other hand, the blue horospheres associated to $\eta$ in the boundary of $\mathbb{R}^2$ are not.} \label{fig: convergence of horospheres}
\end{figure}
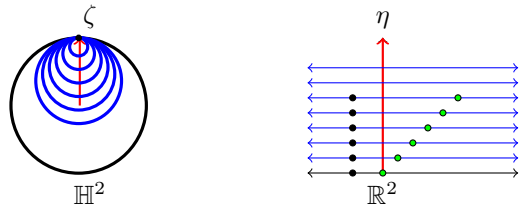

\begin{introthm}\label{main theorem}
Let $X$ be a proper CAT(0) space and let $\partial X$ be its visual boundary. Suppose that $c,c'$ are two distinct geodesic rays such that $c(0)=c'(0)$ and $c$ is $\kappa$-contracting. For any two distinct $\zeta,\eta \in \partial X $ where $\zeta=c(\infty), \eta=c'(\infty)$, if $h$ is a horofunction associated to the geodesic ray $c$ and $H_k=h^{-1}(-k)$ , then we have the following:
\begin{enumerate}
    \item Every sequence $x_k$ with $h(x_k) \rightarrow -\infty $, must satisfy $x_k \rightarrow \zeta.$ In particular, if $\{H_k\}_{k \in \mathbb{N}}$ is a sequence of $h$-horospheres and $x_k \in H_k,$ then $x_k \rightarrow \zeta.$ That is to say, any sequence of $h$-horospheres is convergent.

    \item If $B$, $B'$ are horoballs centered at $\zeta,\eta$ respectively, then $B \cap B'$ is bounded.
    \item If $h'$ is a horofunction associated to $c'$, then $h(c'(s)) \rightarrow \infty$ and $ h'(c(t))  \rightarrow \infty$,   as $s,t \rightarrow \infty.$
 
\end{enumerate}

\end{introthm}

\newpage
We remark that part (1) of Theorem \ref{main theorem}, extends Theorem 6.18 in \cite{Zal18} from the settings of contracting geodesic rays to sublinearly contracting ones. As an application of Theorem \ref{main theorem}, we obtain the following.

\begin{introcor}\label{introcor:visibility} Let $X$ be a proper CAT(0) space and let $\partial X$ be its visual boundary. Fix two distinct points $\zeta, \eta \in \partial X $ where $\zeta=c(\infty)$ and $c$ is sublinearly contracting. There is a geodesic line $l:(-\infty, \infty) \rightarrow \mathbb{R}$ with $l(-\infty)=\zeta$ and $l(\infty)=\eta.$ That is to say, sublinearly contracting geodesic rays define visibility points in the visual boundary $\partial X.$

\end{introcor}

 The above Corollary \ref{introcor:visibility} appears in the article \cite{QingZalloum1} by Qing and the author, however, our argument in that paper is incomplete as it relies on Corollary 9.9 of Chapter \RNum{2}.9  in \cite{BH1} which doesn't hold in the full generality of proper CAT(0) spaces. More precisely, Corollary 9.9 of Chapter \RNum{2}.9  in \cite{BH1} assumes that the angle between any two points (see Definition 9.4 of Chapter \RNum{2}.9  in \cite{BH1}) in the visual boundary is realized as an angle between two geodesic rays emanating from a certain point $x_0$ in the space, which may not always be the case. Therefore, the visibility and the existence of a rank-one isometry results will both be removed in the subsequent updated arXiv version of \cite{QingZalloum1}. The mistake above was pointed out to the author by Merlin Incerti-Medici.

Visibility of the sublinear boundary has a wide range of applications. For example, it's used in a current work-in-progress of Merlin Incerti-Medici and the author to show that sublinearly contracting boundaries of certain CAT(0) cube complexes (the ones with a factor system) continuously inject in a the Gromov's boundary of a $\delta$-hyperbolic space. An immediate consequence of the Corollary \ref{introcor:visibility} is the following.

  \begin{introcor}\label{introcor:rank1} Let $G$ act geometrically on a proper CAT(0) space $X$. If $X$ contains a sublinearly contracting geodesic ray, then $G$ contains a rank one isometry.
 
\end{introcor}

\subsection*{Outline of the paper} Section 2 contains some background material on visual boundaries and horofunctions of CAT(0) spaces, it also contains some preliminaries on sublinearly contracting geodesic rays. Section 3 contains the proof of Theorem \ref{main theorem}, Corollaries \ref{introcor:visibility} and \ref{introcor:rank1}.

\subsection*{Acknowledgement} The author is also thankful to Merlin Incerti-Medici for pointing out the mistake discussed in the introduction and to Yulan Qing for multiple fruitful discussions.

\section{Preliminaries}

\subsection{CAT(0) spaces and their boundaries}
A proper geodesic metric space $(X, d_X)$ is CAT(0) if geodesic triangles in $X$ are at 
least as thin as triangles in Euclidean space with the same side lengths. To be precise, for any 
given geodesic triangle $\triangle pqr$, consider the unique triangle 
$\triangle \overline p \overline q \overline r$ in the Euclidean plane with the same side 
lengths. For any pair of points $x, y$ on edges $[p,q]$ and $[p, r]$ of the 
triangle $\triangle pqr$, if we choose points $\overline x$ and $\overline y$  on 
edges $[\overline p, \overline q]$ and $[\overline p, \overline r]$ of 
the triangle $\triangle \overline p \overline q \overline r$ so that 
$d_X(p,x) = d_\E(\overline p, \overline x)$ and 
$d_X(p,y) = d_\E(\overline p, \overline y)$ then,
\[ 
d_{X} (x, y) \leq d_{\mathbb{R}^{2}}(\overline x, \overline y).
\] 

For the remainder of the paper, we assume $X$ is a proper CAT(0) space. A metric space $X$ is {\it proper} if closed metric balls are compact. We need the following properties of a CAT(0) space:
\begin{lemma}
 A proper CAT(0) space $X$ has the following properties:
\begin{enumerate}
\item It is uniquely geodesic, that is, for any two points $x, y$ in $X$, 
there exists exactly one geodesic connecting them. Furthermore, $X$ is contractible 
via geodesic retraction to a base point in the space. 
\item The nearest point projection from a point $x$ to a geodesic line $b$ 
is a unique point denoted $x_b$. In fact, the closest point projection map
\[
\pi_b \from X \to b
\]
is Lipschitz. 
\item For any $x \in X$, the distance function $d(x,-)$ is convex. In other words, for any given any geodesic $[x_0,x_1]$ and  $t \in [0,1]$, if $x_t $  satisfies $d(x_0,x_t)=td(x_0,x_1)$ then we must have $d(x,x_t) \leq (1-t)d(x,x_0)+td(x,x_1)$.
\end{enumerate}
\end{lemma}

Now we give two definitions of the \emph{visual boundary}
of a CAT(0) space, both of which are needed in this paper.


 \begin{definition}[space of geodesic rays] \label{def: visual topology}
As a set, the \emph{visual boundary} of $X$, denoted by $\partial X$ is defined to be the collection of equivalence classes of all infinite geodesic rays. Let $b$ and $c$ be two infinite geodesic rays, not necessarily starting at the same point. We define an equivalence relation as follows:  $b$ and $c$ are in the same equivalence class,  if and only if there exists some $\nn \geq 0$ such that $d(b(t), c(t)) \leq \nn$ for all $t \in [0 ,\infty).$ We denote the equivalence class of a geodesic ray $b$ by $b(\infty).$
 \end{definition}

 Notice that by Proposition 8.2 in the CAT(0) boundary section of \cite{BH1}, for each $b$ representing an element of $\partial X$, and for each $x' \in X$, there is a unique geodesic ray $b'$ starting at $x'$ with $b(\infty)=b'(\infty).$ Now we describe the topology of the visual boundary: Fix a base point $\go$ and let $b$ be a geodesic ray starting at $\go$. A neighborhood basis for $b(\infty)$ is given by sets of the form: 
\[U(b(\infty),r, \epsilon):=\{ c(\infty) \in \partial X| \, c(0)=\go\,\,\text{and }\,d(b(t), c(t))<\epsilon \, \, \text{for all }\, t < r \}.\]

 In other words, two geodesic rays are close together in this topology if they have
representatives starting at the same point which stay close (are at most $\epsilon$ apart) for a long
time (at least $r$). Notice that the above definition of the topology on $\partial X$ made a reference to a base point $\go$. Nonetheless, Proposition 8.8 in the CAT(0) boundaries section of \cite{BH1} shows that the topology of the visual boundary is a base point invariant. It's also worth mentioning that when $X$ is proper, then the space $\overline{X}=X \cup \partial X$ is compact.

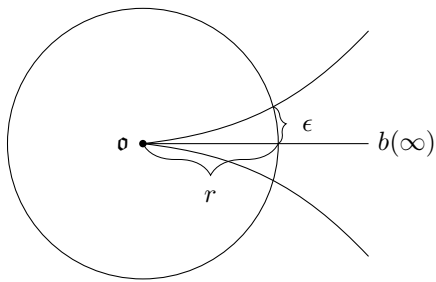
\begin{figure}[h]
\begin{center}
\begin{tikzpicture}[scale=0.6]

\node (x) [circle,fill,inner sep=1pt,label=180:$\go$] at (0,0) {};
\draw [name path=circle] (0,0) circle (3);

\draw [name path=line,thin] (0,0) to [bend right=20] (5,2.5);
\draw [thin] (0,0) to [bend left=20] (5,-2.5);
\draw (0,0) to (5,0) node [right] {$b(\infty)$};

\draw [very thin,
	decorate,
	decoration={brace,mirror,amplitude=12pt}] 
	(0,0) -- (3,0) node [midway,yshift=-20pt] {$r$};

\draw [very thin,
	name intersections={of=line and circle},
	decorate,
	decoration={brace,amplitude=4pt}] 
	(intersection-1) -- (3,0) node [midway,xshift=12pt] {$\epsilon$};

\end{tikzpicture}
\end{center}
\caption{A basis for open sets}
\label{}
\end{figure}

\begin{definition}(Visibility points) For a proper CAT(0) space $X$, a point $\zeta \in \partial X$ is said to be a \emph{visibility point} if for any $\eta \neq \zeta$, there exists a geodesic line $l:(-\infty, \infty) \rightarrow X$ such that $l(-\infty)=\zeta$ and $l(\infty)=\eta.$

\end{definition}

\subsection{Horofunctions, horospheres and horoballs}
\begin{definition}[space of horofunctions]
Let $X$ be any metric space, and let $C(X)$ be the collection of all continuous maps $f:X \rightarrow \mathbb{R}.$ Let $C_{\star}(X)$ denote the quotient of $C(X)$ which is defined by identifying functions which differ by a constant. That is, $C_{\star}(X)= C(X)/ \sim $, where $f \sim g$ if and only if $f(x)=g(x)+K$ for some constant $K \in \mathbb{R}$. For an element $f \in C(X)$, let $\overline{f}$ denote its image in the quotient $C_{\star}(X).$ There is a natural embedding $i:X \rightarrow C_{\star}(X)$ by \[i(x)=\overline{d(x,-)}.\] Define $\hat{X}$ to be the closure of $i(X)$ in $C_{\star}(X)$.

Let $X$ be any metric space. A \emph{horofunction} is a map $h:X \rightarrow \mathbb{R}$ such that $\overline{h} \in \hat{X}-i(X)$. Notice that horofunctions are 1-Lipschitz as they are limits of distance functions. Let $X$ be any metric space and let $c:[0,\infty) \rightarrow X$ be a geodesic ray. Define the \emph{Busemann function} associated to $c$ by: 
\[ b_c(x)=\lim_{t\rightarrow \infty} [d(x,c(t))-t].\]

The above limit exists by the triangle inequality. Notice that Busemann functions are horofunctions. On the other hand, the space of horofunctions is often larger (we will see however that for CAT(0) spaces Busemann functions and horofunctions coincide up to addition by a constant). Let $B$ denote the collection of all Busemann functions over $X$. Define $\overline{B}$ to be the quotient of $B$ denote the quotient of $B$ taken with respect to the equivalence relation that identifies functions which differ by a constant. 

\begin{theorem} [Theorem 8.13 and Proposition 8.20  \cite{BH1}] \label{thm: horofunctions are Busemann}
Let $X$ be a proper CAT(0) space and let $B$ be the space of all Busemann functions over $X$. Fix a base point $\go \in X$ and let $\partial X$ be visual boundary whose elements are uniquely represented by geodesic rays starting at $\go.$ We have the following:

\begin{enumerate}
    \item The class of horofunctions and Busemann functions coincide up to addition by a constant. In other words, every horofunction $h: X \rightarrow \mathbb{R}$ is of the form $b_c+ \nn$ of some Busemann function $b_c$ and a constant $\nn.$
    
    \item The natural map $f:\partial X \rightarrow \overline{B} $ given by $c(\infty) \mapsto \overline{b_c}$ is a homeomorphism. In particular, any two Busemann functions associated to the same boundary point $\zeta$ are the same up to addition by a constant.
\end{enumerate}
\end{theorem}

In light of the above theorem, every geodesic ray uniquely determines a class of horofunctions $[b_c]$ and two horofunctions are in the same class if and only if they differ by a constant. Hence, we may speak of \emph{a horofunction associated to a geodesic ray $c$} by which we mean a map of the form $b_c+\nn$ for some constant $\nn.$

\end{definition}

\begin{definition}[Horospheres, convergence of horospheres and horoballs]
Let $\zeta$ be a point in the visual boundary $\partial X$ and let $h$ be some horofunction associated to $c$ with $c(\infty)=\zeta.$ \begin{itemize}
    \item An \emph{$h$-horosphere} is a set of the form $H_k=h^{-1}(-k)$ for some $k \in \mathbb{R}.$ 
    
    \item A sequence of $h$-horospheres $\{H_k\}_{k \in \mathbb{N}}$ is said to be \emph{convergent} if whenever $x_k \in H_k,$ we have $x_k \rightarrow \zeta.$
    
    \item An \emph{$h$-horoball} is a set of the form $B_{k}:=h^{-1}((-\infty,-k])$ for some $k \in \mathbb{R}.$
\end{itemize} 

\begin{remark} \label{rmk:horospheres}

Notice that for any $\nn \in \mathbb{R}$ both $h$ and $h+\nn$ have the same set of horospheres, therefore, it's more natural to think of horospheres as if they are assigned to classes of horofunctions (where two two horofunctions are the same if and only if they differ by a constant) as opposed being assigned to horofunctions.
\end{remark}
\end{definition}

\begin{lemma}[Proposition 8.22 of Chapter \RNum{2}.8  in \cite{BH1}]\label{lem: convexity of horofunctions} For a proper CAT(0) space $X$, every horofunction $h:X \rightarrow \mathbb{R}$ is a convex 1-Lipschitz map.

\end{lemma}

%
%

\subsection{Sublinearly contracting boundaries of CAT(0) spaces}
%

Let $\kappa \from [0, \infty) \to [1, \infty)$ be a sublinear function that is monotone increasing and concave. That is
\[
\lim_{t \to \infty} \frac{\kappa(t)}{t} = 0.
\]
The assumption that $\kappa$ is increasing and concave makes certain arguments
cleaner, otherwise they are not really needed. One can always replace any 
sub-linear function $\kappa$, with another sub-linear function $\overline \kappa$
so that \[\kappa(t) \leq \overline \kappa(t) \leq \sC \, \kappa(t)\] for some constant $\sC$ 
and $\overline \kappa$ is monotone increasing and concave. For example, define 
\[
\overline \kappa(t) = \sup \Big\{ \lambda \kappa(u) + (1-\lambda) \kappa(v) \ST 
\ 0 \leq \lambda \leq 1, \ u,v>0, \ \text{and}\ \lambda u + (1-\lambda)v =t \Big\}.
\]

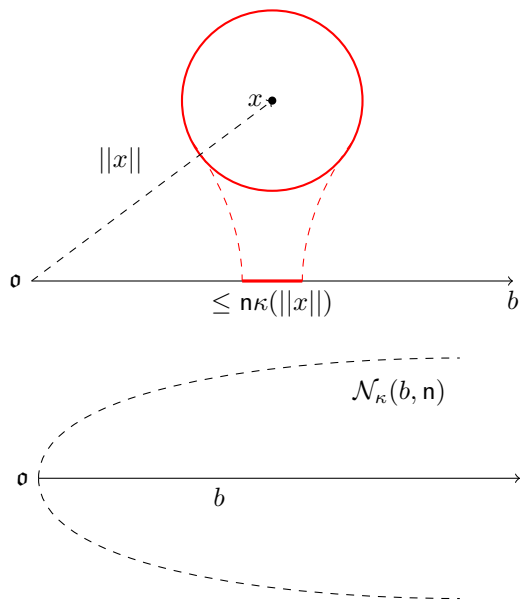
\begin{figure}[!h]
\begin{tikzpicture}[scale=0.8]
\draw[thin,->] (0,0) -- (8,0);

\node[left] at (0,0) {$\go$};

\draw[red, dashed] (5.32,2.3)  .. controls ++(-.3,-.4) and ++(0,1) .. (4.5,0) ;

\draw[red, dashed] (2.68,2.3)  .. controls ++(.3,-.4) and ++(0,1) .. (3.5,0) ;

\draw[very thick, red] (3.5,0) -- (4.5,0);

\node[below] at (4,0) {$ \leq \nn \kappa(||x||)$};

\node[below] at (8,0) {$b$};

\draw[thick,red] (4,3) circle (1.5cm);

\node[ left] at (4,3) {$x$};

\draw[thin,dashed, ->] (0,0) -- (4,3);

\node[ left] at (2,2) {$||x||$};

\draw[thick,fill=black] (4,3) circle (0.05cm);

\end{tikzpicture}
\end{figure}

\begin{definition}[$\kappa$--neighborhood]  \label{Def:Neighborhood} 
For a geodesic ray $b$ and a constant $\nn$, define the \emph{ $(\kappa, \nn)$--neighbourhood} 
of $b$ to be 
\[
\calN_\kappa(b, \nn) = \Big\{ x \in X \ST 
  d_X(x, b) \leq  \nn \cdot \kappa(x)  \Big\}.
\]
 
\end{definition}

\begin{figure}
\begin{tikzpicture}[scale=0.8]
\draw[thin,->] (0,0) -- (8,0);

\node[left] at (0,0) {$\go$};

\node[below] at (3,0) {$b$};

\node[below] at (6,1.8) {$\mathcal{N}_\kappa(b, \nn)$};

\draw[dashed, black] (0,0)  .. controls ++(0,1.5) and ++(-3,0) .. (7,2);

\draw[dashed, black] (0,0) .. controls ++(0,-1.5) and ++(-3,0) .. (7,-2);

\end{tikzpicture}
\end{figure}

\begin{definition}[$\kappa$--contracting] \label{Def:Contracting}
Fix $\go \in X$, for $x \in X$, define $\Norm{x} = d_X(\go, x)$. 
For a geodesic ray $b$ in $X$ with $b(0)=\go$, we say $b$ is \emph{$\kappa$--contracting} if there 
is a constant $\nn$ so that, for every ball $B$ centered at $x$,
\[
B \cap b= \emptyset \quad \Longrightarrow \quad diam (\pi_b(B)) \leq \nn \cdot \kappa(\Norm x).
\]
In fact, to simplify notation, we often drop $\Norm{\cdot}$. That is, for $x \in X$, we define
\[
\kappa(x) := \kappa(\Norm{x}). 
\] We say that $b$ is \emph{sublinearly contracting} if it's $\kappa$-contracting for some sublinear function $\kappa$.
\end{definition}

It is worth pointing out that if a point $\eta$ in the visual boundary has a representative which is a $\kappa$-contracting geodesic ray, then every other geodesic ray representing $\eta$ is also $\kappa$-contracting. In other words, we have the following.

\begin{lemma}
Let $b$ be a $\kappa$-contracting geodesic ray and let $x \in X.$ If $c$ is the unique geodesic ray emanating from $x$ with $c(\infty)=b(\infty),$ then $c$ is also $\kappa$-contracting.
\end{lemma}

\begin{proof}
The proof of this lemma is easy and will be left as an exercise for the reader. The idea is that since $b$ and $c$ are at a finite Hausdorff distance, the are coarsely ``the same". Hence, balls in $X$ project to large sets in $b$ if and only if they project to large sets in $c.$
\end{proof}

The following lemma will be used repeatedly throughout this paper, the conclusion is obvious but we provide a proof for the convenience of the reader. 

\begin{lemma} \label{lem: CAT(0) distance can't be sublinear}
Let $\kappa$ be a sublinear function and let $\nn \geq 0 $. If $f: [0, \infty) \rightarrow [0, \infty)$ is a convex function such that $f(0)=0$ and $f(t_i) \leq \nn \kappa(t_i)$ for an unbounded sequence $t_i \in [0, \infty)$ then $f=0.$
\end{lemma}

\begin{proof}
In order to show this, it suffices to show that for any $\epsilon>0$, we have $f(r)< \epsilon$ for any $r>0.$ Let $\epsilon >0$ be given, fix a random $r>0$. Choose $r'$ large enough so that $\frac{\nn r\kappa(r')}{r'}< \epsilon$, and $f(r') \leq \nn \kappa(r').$ This is possible since $\kappa$ is sublinear and $f(t_i) \leq \nn \kappa(t_i)$ for an unbounded sequence $t_i$. Let $u:[0,1] \rightarrow [0, r']$ parameterizing $[0, r']$ proportional to its arc length. That is, $u(s)=sr'$. Now, since $f$ is a convex function, we have $f(u(s)) \leq (1-s)f(u(0))+sf(u(1))$, and hence

\begin{align*}
f(u(s)) &\leq (1-s)f(u(0))+sf(u(1))\\
                   &=(1-s)f(0)+sf(r') \\
                                  &=0+sf(r') \\
                                  &\leq s\nn \kappa(r').\\
\end{align*}

In particular, taking $s=\frac{r}{r'} \leq 1,$ we get $f(r)=f(u(\frac{r}{r'})) \leq \frac{r}{r'}\nn \kappa(r')<\epsilon.$
\end{proof}

We show that when $X$ is CAT(0), there is a unique geodesic ray in any $(\kappa,\nn)$--neighborhood of a geodesic ray. This is proven in \cite{QR19} in the special case where the geodesic ray is $\kappa$-contracting, here we observe that this is unnecessary.

\begin{lemma}(Uniqueness of geodesics in $\kappa$-neighborhood) \label{lemma: uniqueness of nbhds}
Let $b$ be a geodesic ray in a CAT(0) space $X$ (not necessarily $\kappa$-contracting) starting at $\go$ and let $\nn \geq 0$ be given. The geodesic ray $b$ is the unique geodesic ray starting at $\go$ and contained in the $(\kappa,\nn)$-neighborhood of $b$.
\end{lemma}

\begin{proof}

Suppose that $\kappa$ and $\nn$ are given. Notice that if $c$ is a geodesic ray in some $(\kappa,\nn)$-neighborhood of $b$ with $c(0)=b(0),$ then for each $t$, we have $d(c(t), b) \leq \nn \kappa(t).$ On the other hand, since $b$ is a convex set, and since $c$ is a geodesic ray, the function $f(t)=d(c(t), b)$ is a convex function of $t$ satisfying $f(0)=0$ and $f(t) \leq \nn \kappa(t).$ Hence, by Lemma \ref{lem: CAT(0) distance can't be sublinear}, $f=0$ and $b=c.$

\end{proof}

\begin{notation} Let $X$ be a proper CAT(0) space. We will use the following notation:
\begin{itemize}
    \item For a geodesic ray $b$ and for $x \in X,$ we denote the unique projection of $x$ to $b$ by $x_b.$ 
    
    \item We denote the unique CAT(0) geodesic connecting two points $x,y$ by $[x,y].$
    \item We will use both $b$ and $im(b)$ to denote the image of the geodesic ray $b$ in $X$.
\end{itemize}
\end{notation}
The following lemma states that $\kappa$-contracting geodesics are \emph{$\kappa$-slim}.

\begin{lemma}[Lemma 3.5 \cite{MQZ20}]\label{lem: contracting implies slim} For any proper CAT(0) space, if $b$ is a $\kappa$-contracting geodesic ray, then there exists some $\nn \geq 0$ such that for any $x \in X, y \in b$, we have $d(x_b, [x,y]) \leq \nn\kappa(x_b)$, where $x_b$ is the unique projection of $x$ to $b.$

\end{lemma}

The following lemma provides an equivalent property to being $\kappa$-contracting, it's often easier to work with this property than with the original $\kappa$-contracting definition.

\begin{lemma}[Lemma 2.9 \cite{MQZ20}] \label{lem: contracting implies projection contracting}
A geodesic ray is $\kappa$--contracting if and only if there exists a constant $\nn\geq 0$ such that for any ball $B$ centered at $x$ and disjoint from $b$, we have 
\[ diam(\pi_b(B)) \leq \nn \kappa(x_b). \]

\end{lemma}

\section{Bounded geodesic image property}

We fix $X$ to be a proper CAT(0) space for the rest of the paper and we fix a base point $\go \in X$.

\begin{definition} (Bounded geodesic image)
A geodesic ray $b \subseteq X$ is said to have the \emph{bounded geodesic image property} if for any other geodesic ray $c$ with $c(0)=b(0),$  and any ball $B$ centered at $c(t)$ for some $t$, if $B \cap b =\emptyset,$ we have $diam(\pi_b(B)) \leq \nn$ for some constant $\nn \geq 0$ depending only on $c$ (and not on $t$). Again, we emphasize that the constant $\nn$ is allowed to vary when we vary the geodesics $c$ which we project to $b$.
\end{definition}

The next lemma shows that $\kappa$-contracting geodesic rays satisfy the bounded geodesic image property. All of our subsequent statements hold for all geodesic rays satisfying the bounded geodesic image property and are not specific to $\kappa$-contracting geodesics.
\begin{lemma}($\kappa$-contracting implies bounded geodesic image)\label{lem: contracting implies BGI}
Let $b$ and $c$ be two distinct geodesic rays in a proper CAT(0) space with $b(0)=c(0)=\go.$ If $b$ is $\kappa$--contracting, then:
\begin{enumerate}
    \item The projection of $c$ on $b$, denoted $\pi_b(c)$, is a bounded set.
  
    \item If $B$ is a ball centered at some point in $c$ with $B \cap b =\emptyset$, then $\pi_b(B)$ is also a bounded set. That is, $\kappa$-contracting geodesic rays satisfy the bounded geodesic image property.

\end{enumerate}
\end{lemma}

\begin{proof}
Let $b,c$ be as in the statement of the lemma.

\begin{enumerate}
    \item  By way of contradiction, assume that $c$ has an unbounded projection onto $b$. This yields a sequence $c(s_i)$ such that projections $b(t_i)$ of $c(s_i)$ leave every ball of radius $i$ around $\go$. Applying Lemma \ref{lem: contracting implies slim} to the geodesics $[\go,c(s_i)]$ gives a constant $\nn \geq 0$ so that  $d(b(t_i),c) \leq d(b(t_i), [\go,c(s_i)]) \leq \nn \kappa(t_i),$ for all $i.$ This gives an unbounded sequence $t_i \in [0,\infty)$ so that $d(b(t_i),c) \leq \nn \kappa(t_i).$ Therefore, applying Lemma \ref{lem: CAT(0) distance can't be sublinear} to $f(t)=d(b(t),c)$, we get that $f=0$ and hence $b=c$ which contradicts the assumption. 
    
    \item This is immediate using the first part of this lemma and Lemma \ref{lem: contracting implies projection contracting}.
\end{enumerate}

\end{proof}

\begin{lemma}\label{Uniform over geodesics}
Let $c$ be a geodesic ray with the bounded geodesic image property, and let $c'$ be a distinct geodesic ray with $c(0)=c'(0).$ For any $x \in c' $ and $y \in c$, we have 

\[ d(x, x_c)+d(x_c, y)-\nn-1 \leq d(x,y) \leq d(x, x_c)+d(x_c, y) \]
 for some constant $\nn$ depending only on $c'$.
\end{lemma}

\begin{proof}

Let $c,c',x,y$ be as in the statement of the lemma. The right-hand side inequality holds by the triangle inequality. For the other side, consider a ball $B$ around $x \in c'$ whose radius $r=d(x,x_c)-1$, and let $z$ be the point on $[x,y]$ with $d(x,z)=r$. Notice that since $c$ has the bounded geodesic image property and since $z \in B$, we have that \[d(x_c,z_c) \leq \nn.\]
Notice that

\begin{align*}
d(x,y)-d(x, x_c) &=d(x,z)+d(z,y)-d(x,x_c)\\
                                          &=(d(x,x_c)-1)+d(z,y)-d(x,x_c)\\
                                          &=d(z,y)-1 \\
                                          &\geq d(z_c, y)-1\\
                                          &=d(x_c,y)-d(x_c, z_c)-1 \\
                                          &\geq d(x_c, y)- \nn -1,
\end{align*} 
which yields

\[ d(x,y) \geq d(x, x_c)+d(x_c, y)-\nn-1.\]

\end{proof}

The following is the key to showing that geodesic rays with the bounded geodesic image property (in particular, $\kappa$-contracting geodesic rays) define visibility points in the visual boundary.

\begin{lemma} \label{key to visibility}
Let $c, c'$ be two distinct geodesic rays starting at $\go$ such that $c$ has the bounded geodesic image property. We have the following:

\begin{enumerate}
    \item  If $b_c,b_{c'}$ denote the Busemann functions of $c,c'$ respectively, then we have $b_{c}(c'(s))\rightarrow \infty$ and $b_{c'}(c(t)) \rightarrow \infty$ as $s,t \rightarrow \infty$.
    
    \item If $d, d'$ are geodesic rays with $d(0)=d'(0)$ and $d(\infty)=c(\infty)$, $d'(\infty)=c'(\infty)$, then $b_{d}(d'(s))\rightarrow \infty$ and $b_{d'}(d(t)) \rightarrow \infty$ as $s,t \rightarrow \infty$.
    \item If $h,h'$ are two horofunctions associated to $c,c'$ respectively, then $h(c'(s)) \rightarrow \infty$ and $h'(c(t)) \rightarrow \infty$ as $s,t \rightarrow \infty.$

\end{enumerate}

\end{lemma}

\begin{proof}
Let $c,c',d,d'$ and $\go$ be as in the statement of the lemma. 

\begin{enumerate}
    \item  Set $d_s=d(c'(s),c)$. Since $c$ has the bounded geodesic image property, there exists a constant $\nn$ depending only on $c'$ such that $diam(\pi_c(c')) \leq \nn.$ Hence, we have $d(\pi_{c}(c'(s)), \go)=\nn_s$ where $0 \leq \nn_s \leq \nn.$ By the triangle inequality, we have that $d_s \geq ||c'(s)||-\nn_s=s-\nn_s.$ For a fixed $s$, let $d_{s,t}=d(c'(s),c(t)).$ Notice that by Lemma \ref{Uniform over geodesics}, we have that

\begin{equation}\label{E1}
d_{s,t} \geq d_s +t -\nn_s - \nn-1 \tag{*}
\end{equation} or
$$d_{s,t} -t \geq (d_s- \nn)-(\nn_s +1).$$ Since the above is independent of $t,$ we get 
$$b_{c}(c'(s)) \geq (d_s- \nn)-(\nn_s +1).$$ By the triangle inequality, we have

$$b_{c}(c'(s)) \geq (||c'(s)||-\nn_s- \nn)-(\nn_s +1)=s-2\nn_s-\nn-1,$$ which implies that $b_{c}(c'(s)) \rightarrow \infty$ as $s \rightarrow \infty,$ which finishes the first part of the lemma. While the projection of $c$ to $c'$ isn't apriori bounded, the argument showing $b_{c'}(c(t)) \rightarrow \infty$ is the same. For completeness, we give the proof. Using equation (\ref{E1}) above, we have

 \begin{align*}
d_{s,t} &\geq d_s +t -\nn_s - \nn-1\\
&\geq s-\nn_s+t-\nn_s-\nn-1. 
\end{align*} 
Hence, $d(c'(s),c(t))-s=d_{s,t}-s \geq t-2\nn_s-\nn-1$. Now, fixing $t$ and letting $s \rightarrow \infty,$ gives that $$b_{c'}(c(t)) \geq t-2\nn_s-\nn-1.$$ Therefore, we have $b_{c'}(c(t)) \rightarrow \infty$ as $t\rightarrow \infty.$

\item This follows immediately since $b_d=b_c +\nn_1$ and $b_{d'}=b_d+\nn_2$ for some constants $\nn_1,\nn_2 \in \mathbb{R}$ by Theorem \ref{thm: horofunctions are Busemann}.

\item Similar to part (2), part (3) follows immediately as $h=b_c +\mm_1$ and $h'=b_{c'}+\mm_2$ for some constants $\mm_1,\mm_2 \in \mathbb{R}$ by Theorem \ref{thm: horofunctions are Busemann}.

\end{enumerate}

\end{proof}

We remark that since $\kappa$-contracting geodesic rays satisfy the bounded geodesic image property (Lemma \ref{lem: contracting implies BGI}), the above lemma gives part (3) of Theorem \ref{main theorem}.

The following lemma states that horospheres corresponding to a geodesics with the bounded geodesic image property are convergent.

\begin{lemma}\label{lem: BGI gives convergence of horospheres}  Let $X$ be a proper CAT(0) space and let $\zeta \in \partial X$. Suppose that $\zeta=c(\infty)$ for some geodesic ray $c$ with the bounded geodesic image property. If $h$ is a horofunction associated to $c$, then any sequence of $h$-horospheres must converge to $\zeta$.

\end{lemma}

\begin{proof}
Let $\zeta$ be a point in the visual boundary and let $c$ be a geodesic ray with the bounded geodesic image property starting at $\go$ with $c(\infty)=\zeta$. Every horofunction $h$ associated to $c$ is of the form $h=b_c + \nn$ for some constant $\nn \geq 0$. In light of Remark \ref{rmk:horospheres}, as horospheres of $b_c$ and $b_c+\nn$ are the same, it suffices to prove the assertion for the special case where $\nn=0.$ If $H_n$ is a sequence of $h$-horospheres, and $x_n \in H_n$, then, by definition of $H_n$, we have $h(x_n) \rightarrow - \infty.$ This implies that $x_n$ is unbounded, and hence, using the fact that $X$ is proper, some subsequence, $c''_{n_k}=[\go,x_{n_k}] \rightarrow c'$ for some geodesic ray $c'$ with $c'(0)=\go$ (recall the definition of convergence in the visual boundary following Definition \ref{def: visual topology}). We claim that $c'(\infty)=c(\infty)=\zeta$. This follows almost immediately using Lemma \ref{key to visibility} and the fact that $h$ is a convex 1-Lipschitz map (Lemma \ref{lem: convexity of horofunctions}). More precisely, suppose $c'(\infty) \neq c(\infty)$, since $h(x_{n_k}) \rightarrow -\infty,$ there exists some $m_1$ such that for all $k \geq m_1,$ we have $h(x_{n_k}) \leq -10$. On the other hand, by Lemma \ref{key to visibility}, we have $h(c'(s)) \rightarrow \infty$, and hence, there exists some $s_0$ such that $h(c'(s)) \geq 10$ for all $s \geq s_0.$ Now, since $c''_{n_k} \rightarrow c',$ there exists $m \geq m_1$ such that if $k \geq m,$ we have $c''_{n_k} \in U(c'(\infty),s_0,1).$ That is to say, $d(c''_{n_k}(t), c'(t))<1$ for all $t \in [0,s_0].$ In particular,  $d(c''_{n_k}(s_0), c'(s_0))<1.$ Using Lemma \ref{lem: convexity of horofunctions}, since the function $h$ is 1-Lipschitz, we have $|h(c''_{n_k}(s_0)-h(c'(s_0))|\leq d(c''_{n_k}(s_0), c'(s_0))<1 $. Hence, we have 
$$h(c''_{n_k}(s_0)) \geq h(c'(s_0))-1 \geq 10-1=9.$$ Since $m$ was chosen to be larger than or equal to $m_1,$ we also have $h(x_{n_k}) \leq -10$ for all $k \geq m.$ Consider the single geodesic $c''_{n_m}=[\go, x_{n_m}]$, for the points $\go, c''_{n_m}(s_0), x_{n_m} \in [\go,x_{n_m}]$, the values of the function $h$ satisfy $h(\go)=0,\,h(c''_{n_m}(s_0)) \geq 9$ and $h (x_{n_m}) \leq -10$ which contradicts convexity of $h$, therefore $c'(\infty)=c(\infty).$ This shows that every convergent subsequence of $x_n$ must converge to $c,$ and since there exists at least one convergent subsequence (as $X$ is proper), we conclude $x_n \rightarrow c.$

\end{proof}

\begin{remark}
Notice that convergence of horospheres is a hyperbolicity phenomena. For example, in $\mathbb{R}^2$ if one takes the geodesic ray $c$ to be the positive $y$-axis, then the associated horospheres, which are all  hyperplanes perpendicular to $c$, do not converge since we have two different sequences (the one in black and the one in green in Figure \ref{horospheres in plane do not converge}) living on the same set of horospheres yet defining different points in $\partial \mathbb{R}^2$. However, as shown in Figure \ref{horospheres in hyperbolic plane do converge}, if we consider horoshperes centered at $\zeta$ in the Poincare disk model, we can see that any convergent sequence living on those horospheres must converge to $\zeta$.
\end{remark}

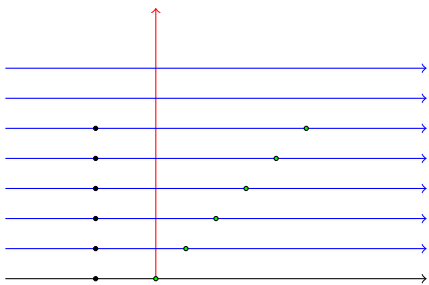
\begin{figure}
    \centering

\begin{tikzpicture}[scale=.4]
\draw[thin,->] (-5,0) -- (9,0);
\draw[thin, red, ->] (0,0) -- (0,9);

\draw[thin, blue, ->] (-5,1) -- (9,1);
\draw[thin, blue, ->] (-5,2) -- (9,2);
\draw[thin, blue, ->] (-5,3) -- (9,3);
\draw[thin, blue, ->] (-5,4) -- (9,4);
\draw[thin, blue, ->] (-5,5) -- (9,5);
\draw[thin, blue, ->] (-5,6) -- (9,6);
\draw[thin, blue, ->] (-5,7) -- (9,7);

\draw [fill=green] (0,0) circle (.07cm);
\draw [fill=green] (1,1) circle (.07cm);
\draw [fill=green] (2,2) circle (.07cm);
\draw [fill=green] (3,3) circle (.07cm);
\draw [fill=green] (4,4) circle (.07cm);
\draw [fill=green] (5,5) circle (.07cm);

\draw [fill=black] (-2,0) circle (.07cm);
\draw [fill=black] (-2,1) circle (.07cm);
\draw [fill=black] (-2,2) circle (.07cm);
\draw [fill=black] (-2,3) circle (.07cm);
\draw [fill=black] (-2,4) circle (.07cm);
\draw [fill=black] (-2,5) circle (.07cm);

\end{tikzpicture}

\caption{ The blue lines are the horospheres of $b_c$ where $c$ is the positive red ray in $\mathbb{R}^2$. Horospheres of $\mathbb{R}^2$ do not converge.}\label{horospheres in plane do not converge}

\end{figure}

\begin{figure}
    \centering

\begin{tikzpicture}[scale=.4]

\draw[very thick,red] (0,0) circle (3cm);


\draw[very thick,blue] (0,2.6) circle (.4cm);
\draw[very thick,blue] (0,2.3) circle (.7cm);
\draw[very thick,blue] (0,2) circle (1cm);
\draw[very thick,blue] (0,1.7) circle (1.3cm);
\draw[very thick,blue] (0,1.4) circle (1.6cm);
\draw[very thick,blue] (0,1.1) circle (1.9cm);

\draw[thick,fill=black] (0,3) circle (.1cm);

\node[above] at (.5,3) {$\zeta$};

\end{tikzpicture}

\caption{Horospheres in $\mathbb{H}^2$ are convergent.} \label{horospheres in hyperbolic plane do converge}
\end{figure}
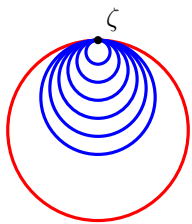

In particular, the conclusion holds for $\kappa$-contracting geodesic rays:

\begin{corollary}[Sublinearly contracting horospheres are convergent] Let $X$ be a proper CAT(0) space and let $\zeta \in \partial X$. Suppose that $\zeta=c(\infty)$ for some geodesic ray $c$ which is $\kappa$-contracting. If $h$ is a horofunction associated to $c$, then any sequence of $h$-horospheres must converge to $\zeta$.

\end{corollary}

\begin{proof}
Using Lemma \ref{lem: contracting implies BGI}, $\kappa$-contracting geodesics satisfy the bounded geodesic image property. Furthermore, Lemma \ref{lem: BGI gives convergence of horospheres} shows the desired statement for geodesic rays with the bounded geodesic image property.
\end{proof}
We remark that the above corollary is part (1) of Theorem \ref{main theorem}. The following lemma states that for any horofunction $h$, if $c,c'$ are geodesic rays with $c(\infty)=c'(\infty),$ then the restriction of $h$ to $im(c)$ is bounded above if and only if the restriction of $h$ to $im(c')$ is bounded above.

\begin{lemma}\label{bounded above independence of representative}
Let $X$ be a proper CAT(0) space and let $h$ be horofunction. If $c,c'$ are geodesic rays with $c(\infty)=c'(\infty),$ then 

\begin{center}
    $h|_{im(c)}$ is bounded above iff   $h|_{im(c')}$ is bounded above.
\end{center}
\end{lemma}

\begin{proof}
The proof is immediate as horofunctions are $1$-Lipschitz (Lemma \ref{lem: convexity of horofunctions}).
\end{proof}

\begin{lemma}\label{lem: bounded diameter}
Let $c, c'$ be distinct geodesic rays starting at $\go$ and let $h,h'$ be two  horofunctions corresponding to $c,c'$ respectively. If $B$ is an $h$-horoball and $B'$ is an $h'$-horoball, then $B \cap B'$ is a bounded set provided that $c$ has the bounded geodesic image property.
\end{lemma}

\begin{proof}
 Since the distance function in a CAT(0) space is convex, every horoball must be a convex set, and hence $B \cap B'$ is convex. If $B \cap B'$ is an unbounded set, then, using the fact that $X$ is proper, we get an infinite sequence of points $x_n \in B \cap B'$ such that for $x_0 \in B\cap B'$, the sequence of geodesic segments $[x_0, x_n]$ has a subsequence converging to some geodesic ray $b$ in $B \cap B'.$ Notice that by definition of $B$ and $B'$, the  horofunctions $h$ and $h'$ are both bounded above when restricted to the geodesic ray $b.$ By Lemma~\ref{bounded above independence of representative}, if $b'$ is the unique geodesic ray starting at $\go$ with $b(\infty)=b'(\infty),$ then $h$ and $h'$ are both bounded above when restricted to $b'.$ Using Lemma \ref{key to visibility}, we get that $c(\infty)=c'(\infty)=b'(\infty)$ which is a contradiction.
\end{proof}

The following is part (2) of Theorem \ref{main theorem}.

\begin{corollary}\label{cor: bounded diameter}
Let $c, c'$ be distinct geodesic rays starting at $\go$ and let $h,h'$ be two horofunctions corresponding to $c,c'$ respectively. If $B$ is an $h$-horoball and $B'$ is an $h'$-horoball, then $B \cap B'$ is a bounded set provided that $c$ is $\kappa$-contracting.
\end{corollary}
\begin{proof}
This is immediate using Lemma \ref{lem: bounded diameter} and Lemma \ref{lem: contracting implies BGI}.
\end{proof}

  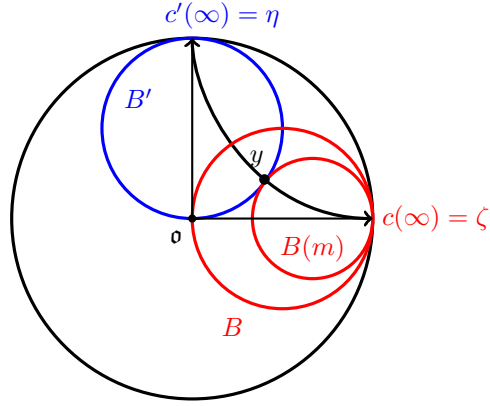
\begin{figure}\label{Hyperbolic case}
    \centering

\begin{tikzpicture}[scale=.8]

\draw[very thick,black] (0,0) circle (3cm);


\draw[very thick,black, <->] (0,3) .. controls ++(0,-1) and ++({-2*cos(0)},{-1.1*sin(0)}) ..({3*cos(0)},{2*sin(0)});



\draw[very thick,blue] (0,1.5) circle (1.50cm);

\draw[very thick,red] (2,0) circle (1cm);

\draw[very thick,red] (1.5,0) circle (1.5cm);

\node[left, red] at (1,-1.8) {$B$};

\node[left, red] at (2.7,-.5) {$B(m)$};

\node[left, blue] at (-.5,2) {$B'$};

\node[above, blue] at (.5,3) {$c'(\infty)=\eta$};

\node[right, red] at (3,0) {$c(\infty)=\zeta$};

\node[right, black] at (.8,1) {$y$};

\draw[thick,fill=black] (1.2,.65) circle (.07cm);

\draw[thick,->] (0,0) -- (3,0);
\draw[thick,->] (0,0) -- (0,3);

\node[left] at (0,-.3) {$\go$};

\draw[thick,fill=black] (0,0) circle (.05cm);

\end{tikzpicture}

\caption{Proof of visibility.} \label{fig: proof of visibility}

\end{figure}
\begin{corollary}\label{cor: The sublinear boundary is a visibility space}
Let $X$ be a proper CAT(0) space and let $\zeta, \eta$ be two distinct points of the visual boundary $\partial X$. If there exists a geodesic ray $c$ with the bounded geodesic image property satisfying $c(\infty)=\zeta$,  then there exists a geodesic line $l:(-\infty, \infty) \rightarrow X$ with $l(-\infty)=\zeta$ and $l(\infty)=\eta.$

\end{corollary}
\begin{proof} 
We remark that the following argument, while done in slightly different settings, is exactly the argument $((4)\Rightarrow (1))$ Proposition 9.35 in Chapter \RNum{2}.9  of \cite{BH1}. For completeness, we provide the proof.

Let $\zeta, \eta$ and $c$ be as in the statement of the theorem. Let $c'$ be the unique geodesic ray representing $\eta$ with $c(0)=c'(0)=\go$, in other words, we have $c(\infty)=\zeta$ and $c'(\infty)=\eta.$ Notice that if we set $B=b_{c}^{-1}((-\infty,0])$ and $B'=b_{c'}^{-1}((-\infty,0])$, then $B \cap B'$ has a bounded diameter, say $K,$ by Lemma \ref{lem: bounded diameter}. Also, if $a \geq 0,$ and we set \[B(a)=b_{c}^{-1}((-\infty,-a]),\]
 then 
 
\[B(a) \cap B' \subseteq B \cap B' \text{ for all } a \geq 0.\]
 Thus, the diameter of $B(a) \cap B'$ is bounded bounded by $K$, for each $a \geq 0.$ Also, notice that if $a >K$, then $B(a) \cap B'= \emptyset$, as if $y \in X$ is such that $b_c(y) \leq -a$, then 
 \[K<a \leq |b_{c}(y)|=|b_{c}(y)-b_{c}(\go)| \leq d(\go,y).\]
  Hence, \[
 B(a) \cap B' =\emptyset\]
  for any $a>K.$ This yields that $m=$Sup$\{a \geq 0 | B(a) \cap B' \neq \emptyset\}$ is a well defined quantity. Furthermore, since $X$ is proper, the supremum is realized. That is, there exists $y \in B(m) \cap B'$, see Figure \ref{fig: proof of visibility}. Hence, if we let $d, d'$ be the unique geodesic rays starting at $y$ with $c(\infty)=d(\infty)$ and $c'(\infty)=d'(\infty)$, then by our choice of $m$, we have \[b_c(d(t))=-t-m \text{ and } b_{c'}(d'(t))=-t.\] We claim that the union of the geodesic rays $d$ and $d'$ is a geodesic line. It suffices to show that there exists a large enough $t>0$ such that \[ 2t \geq d(d(t), d'(t)) \geq 2t.\]
   Notice that the left hand side of the inequality holds by the triangle inequality, so we need only to show that $d(d(t), d'(t)) \geq 2t.$ Let $[d(t), d'(t)]$ be the unique geodesic segment connecting $d(t)$ to $d'(t)$. Since $b_c (d(t))=-t-m$, Lemma \ref{key to visibility} and the intermediate value theorem imply the existence of some $z \in [d(t), d'(t)]$ such that $b_c (z)=-m$. Therefore, since $z$ lives in the horospheres $b^{-1}_c (-m)$, and by our choice of $m$,  we have $b_{c'} (z) \geq 0$. Hence, 
\begin{align*}
d(d(t), d'(t)) &=d(d(t), z)+d(z,d'(t))\\
                                  & \geq |b_{c}(d(t))-b_{c}(z)|+|b_{c'}(z)-b_{c'}(d'(t))| \\
                                  &\geq 2t.
\end{align*}
\end{proof}
Since $\kappa$-contracting geodesic rays satisfy the bounded geodesic image property (Lemma \ref{lem: contracting implies BGI}), we get the following.

\begin{corollary}\label{cor: strong visibility}
If $\zeta \in \partial X$ is such that $\zeta=c(\infty)$ for some $\kappa$-contracting geodesic ray $c$, then $\zeta$ is a visibility point of $\partial X$. That is to say, if $\zeta, \eta$ are distinct points in $\partial X$ such that $\zeta=c(\infty)$, for some $\kappa$-contracting geodesic ray $c,$ then there exists a geodesic line $l:(-\infty, \infty) \rightarrow X$ with $l(-\infty)=\zeta$ and $l(\infty)=\eta.$

\end{corollary}

\begin{proof}
Using Lemma \ref{lem: contracting implies BGI}, $\kappa$-contracting geodesic rays satisfy the bounded geodesic image property, and Corollary \ref{cor: The sublinear boundary is a visibility space} gives the desired statement for all such geodesic rays.

\end{proof}

We remark that Corollary \ref{cor: strong visibility} gives Corollary \ref{introcor:visibility} in the introduction.

The following is an immediate consequence of Ballmann and Buyalo [\cite{Ballmann2008}, Proposition 1.10].

\begin{corollary}\label{Cor: rank one isometries exist}
Let $G$ act geometrically on a proper CAT(0) space $X$ and let $\partial X$ denote the visual boundary of $X$. If $\partial X$ contains a visibility point, then $G$ contains a rank one isometry.
\end{corollary}

Hence, the previous two corollaries imply the following.

\begin{corollary}
Let $G$ be a group acting geometrically on a proper CAT(0) space $X$. If $\partial X$ contains a point $\zeta$  such that $\zeta=c(\infty)$ for some $\kappa$-contracting geodesic ray $c$, then $G$ contains a rank one isometry.
\end{corollary}

\subsection*{Data Availability Statement}

Data sharing not applicable to this article as no data sets were generated or analysed during the current study.

\bibliography{bibliography}{}
\bibliographystyle{plain}

\end{document}